\documentclass[reqno,b5paper]{amsart}
\usepackage{amsmath}
\usepackage{amssymb}
\usepackage{amsthm}
\usepackage{enumerate}
\usepackage[mathscr]{eucal}
\theoremstyle{plain}
\newtheorem{thm}{Theorem}[section]
\newtheorem{cor}[thm]{Corollary}
\newtheorem{lem}[thm]{Lemma}
\newtheorem{Example}{Example}[section]
\newtheorem{note}{Note}[section]

\theoremstyle{definition}
\newtheorem{defn}{Definition}[section]
\newtheorem{rem}{Remark}[section]

\begin{document}

\setcounter {page}{1}
\title{  On $I$-statistical cluster point of double sequences}

\author[P. Malik AND A. Ghosh]{ Prasanta Malik* and Argha Ghosh* \ }
\newcommand{\acr}{\newline\indent}
\maketitle
\address{{*\,} Department of Mathematics, The University of Burdwan, Golapbag, Burdwan-713104,
West Bengal, India.
                Email: pmjupm@yahoo.co.in., buagbu@yahoo.co.in\acr
          \\}

\maketitle
\begin{abstract}In this paper we are concerned with the recent summability notion of 
$I$-statistically pre-Cauchy real double sequences in line of Das et.
al. [6] as a generalization of $I$-statistical convergence. Here we introduce the notion of double 
$I$-natural density and present some interesting properties of 
$I$-statistically pre-Cauchy double sequences of real numbers. Also in this paper we investigate the notion of $I$-statistical cluster point of double sequences in finite dimensional normed space.
\end{abstract}
\author{}
\maketitle


\section{\textbf{Introduction}}

The usual notion of convergence of real sequences was extended
to statistical convergence independently by Fast [7] and Schoenberg[14] 
based on the notion of natural density. Statistical 
convergence is one of the most active area of study
in the summability theory.

The concept of statistical convergence was further extended to 
$I$-convergence [12] using the notion of ideals of $\mathbb{N}$ 
and to $I$-statistical convergence [1] ( see also [4], [5], [8], [13] ). 

The concept of statistically pre-Cauchy sequences was introduced by 
Connor et. al. in [2]. They proved that every statistically convergent 
sequences are statistically pre-Cauchy but the converse holds 
under certain conditions. Recently Das et. al. in [6] introduced the 
notion of $I$-statistical pre-Cauchy condition for real sequences 
and established some basic properties of this notion and 
in [16] Yamanci et. al. studied the same notion for real double 
sequences. In [6] Das and Savas also introduced the notion 
of $I$-natural density. As a natural consequence in this paper we 
study some interesting properties of real $I$-statistically pre-Cauchy 
double sequences and introducing the notion of double 
$I$-natural density  we establish relationship between double $I$-natural density and $I$-statistical pre-Cauchy condition for double sequences of real numbers. Also we introduced the notion of $I$-statistical cluster point of double sequences in  finite dimensional normed space and prove some results analogous to the results of \cite{Ma}.\\

\section{\textbf{Basic Definitions and Notations  }}
Throughout the paper $\mathbb{N}$ denotes the set of all positive integers and $\mathbb{R}$ denotes the set of all real numbers.

\begin{defn}$[9]$   
Let $K\subset\mathbb{N}\times\mathbb{N}$ and $K(n,m)$ be the
number of $(j,k)\in K$ such that $j\leq n, k\leq m$. If the sequence
$\{\frac{K(n,m)}{nm}\}_{n,m\in\mathbb{N}}$ has a limit in
Pringsheim's sense (See [10]), then we say that $K$ has the double natural
density and it is denoted by
\begin{center}
$d_{2}(K)=\underset{n\rightarrow\infty}{\underset{m\rightarrow\infty}{\lim}}\frac{K(n,m)}{nm}$.
\end{center}
\end{defn}

\begin{defn} $[9]$ A double sequence $x=\{x_{jk}\}_{j,k\in\mathbb{N}}$ of real numbers is said to be statistically convergent to
$\xi \in \mathbb{R}$ if for every $\epsilon> 0$, we have
$d_{2}(A(\epsilon))=0$ where $A(\epsilon)=\{(j,k)\in
\mathbb{N}\times\mathbb{N}; |x_{jk}-\xi|\geq\epsilon\}$. In
this case we write
$st-\underset{k\rightarrow\infty}{\underset{j\rightarrow\infty}{\lim}}x_{jk}=\xi$.
\end{defn}

\begin{defn}$[9]$  A double sequence $x=\{x_{jk}\}_{j,k\in\mathbb{N}}$  of real numbers is said to be statistically
Cauchy if for every $\epsilon\geq 0$, there exist natural
numbers $N=N(\epsilon)$ and $M=M(\epsilon)$ such that for
all $j,p\geq N$ and $k,q\geq M$,
\begin{center}
$d_{2}(\{(j,k)\in\mathbb{N}\times\mathbb{N}: |x_{jk} - x_{pq}|\geq\epsilon\})=0$
\end{center}
\end{defn}

We now recall definitions of ideal and filter on a nonempty set.

\begin{defn}
 Let $X\neq\phi $. A class $ I $ of subsets of X is said to be
an ideal in X provided, $I$ satisfies the conditions:
\\(i)$\phi \in I$,
\\(ii)$ A,B \in I \Rightarrow A \cup B\in I,$
\\(iii)$ A \in I, B\subset A \Rightarrow B\in I$.
\end{defn}

An ideal $I$ in a non-empty set $X$ is called non-trivial if X $\notin I.$

\begin{defn}
 Let $X\neq\phi  $. A non-empty class $\mathbb F $ of subsets of X is
said to be a filter in X provided that:
\\(i)$\phi\notin \mathbb F $,
\\(ii) $A,B\in\mathbb F \Rightarrow A \cap B\in\mathbb F,$
\\(iii)$ A \in\mathbb F, B\supset A \Rightarrow B\in\mathbb F$.
\end{defn}
 
\begin{defn}
 Let $I$ be a non-trivial ideal in a non-empty set $ X$.
 Then the class $\mathbb F(I)$$ = \left\{M\subset X : \exists A \in I ~such~~ that ~M = X\setminus A\right\}$
is a filter on X. This filter $\mathbb F(I) $ is called the filter associated with $I$.
\end{defn}
 
A non-trivial ideal $I$ in $X(\neq \phi)$ is called admissible if $\left\{x\right\} \in I$ for each $x \in X$.

Throughout the paper we take $I$ as a non-trivial admissible ideal in $\mathbb{N} \times \mathbb{N} $. 

\begin{defn}$[12]$
A double sequence $x=\left\{x_{jk}\right\}_{j,k\in \mathbb{N}}$ of real numbers is said to converge to $\eta\in \mathbb{R}$ with respect to the ideal $I$, if for every $\epsilon > 0$ the set $A(\epsilon)=\left\{(m,n):\left|x_{mn}-\eta\right|\geq \epsilon\right\}\in I.$
\end{defn}

\begin{defn}$[1]$
 A double sequence $\{x_{jk}\}_{j,k\in \mathbb{N}}$ of real numbers is  $I$-statistically convergent to L, and we write $x_{jk}\stackrel{I^s}{\rightarrow}L$, provided that for $\epsilon> 0$ and $\delta> 0$
\begin{center}
$\left\{\left(m,n\right)\in\mathbb{N}\times\mathbb{N} :\frac{1}{mn}\left|\left\{\left(j,k\right):\left|x_{jk}-L\right|\geq\epsilon,j\leq m,k\leq n\right\}\right|\geq\delta\right\}\in I$
\end{center}
\end{defn}

Now for fixed $p_1,q_1,p_2,q_2\in\mathbb{N}$, we consider the ordered pairs $(j,k)_{p_1q_1}$ 
and $(j,k)_{p_2q_2}$ as different elements of $\mathbb{N}\times\mathbb{N}$ provided $(p_1,q_1)\neq (p_2,q_2)$.

\begin{defn}
 A double sequence $\{x_{jk}\}_{j,k\in \mathbb{N}}$ of real numbers is said to be $I$-statistically pre-Cauchy if for any $\epsilon> 0 $ and $\delta> 0$
\begin{center}
$\left\{(m,n)\in\mathbb{N}\times\mathbb{N}:\frac{1}{m^2n^2}\left|\left\{(j,k)_{pq}:\left|x_{jk}-x_{pq}\right|\geq\epsilon;j,p\leq m;k,q\leq n\right\}\right|\geq\delta\right\}\in I$.
\end{center}
\end{defn}

\begin{defn}$[11]$
 Let $x=\left\{x_{jk}\right\}_{j,k\in\mathbb N}$ be a double sequence of real numbers and let $u_n=sup\left\{x_{jk}:j,k\geq n\right\}, n \in \mathbb{N}$. Then Pringsheim limit superior of $x$ is defined as follows :\\ (i) if $u_n=+\infty$ for each n, then P-$limsup~ x =\infty$,\\ (ii) if $u_n<\infty $ for some n, then P-$limsup ~x=\underset{n}{inf} ~ u_n$.\\ Similarly, let $l_n=inf\left\{x_{jk}:j,k\geq n\right\}, n \in \mathbb{N} $. Then Pringsheim limit inferior of $x$ is defined as follows :\\ (i) if $l_n=-\infty$ for each n, then P-$liminf~ x =-\infty$,\\ (ii) if $l_n>-\infty $ for some n, then P-$liminf ~x= \underset{n}{sup}~ l_n$.
\end{defn}

\section{\textbf{Main Results  }}
First we present an interesting property of $I$-statistically pre-Cauchy  double sequences of real numbers in
line of Theorem 2.4 [6].
\begin{thm}
Let $x=\left\{x_{jk}\right\}_{j,k\in \mathbb N}$ be a double sequence of real numbers and $(\alpha,\beta)$ is an open interval such that $x_{jk}\notin(\alpha,\beta),$ for all $(j,k)\in \mathbb{N}\times\mathbb{N}$. We write $A=\left\{(j,k):x_{jk}\leq\alpha\right\}$ and $B=\left\{(j,k):x_{jk}\geq\beta\right\}$ and further assume that the following property is satisfied
\begin{center}
$P-limsup D_{mn}(A)~~ - ~~ P-liminf D_{mn}(A)<r$
\end{center}
for some $0\leq r\leq 1.$ If $x$ is $I$-statistically pre-Cauchy then either $I-\underset{n\rightarrow\infty}{\underset{m\rightarrow\infty}{\lim}}D_{mn}(A)=0$ or $I-\underset{n\rightarrow\infty}{\underset{m\rightarrow\infty}{\lim}}D_{mn}(B)=0$, where $D_{mn}(A)=\frac{1}{mn}\left|\left\{(j,k)\in A:j\leq m, k\leq n\right\}\right|$.
\end{thm}
\begin{proof}
Here $B=\mathbb{N}\times\mathbb{N}\setminus A$ and so $D_{mn}(B)=1-D_{mn}(A)$ for all $(m,n)\in\mathbb{N}\times\mathbb{N}.$ To complete the proof it is sufficient to show that either $I-\underset{n\rightarrow\infty}{\underset{m\rightarrow\infty}{\lim}}D_{mn}(A)=0$ or 1. Note that
\begin{equation}
\chi_{A\times B}((j,k),(p,q))\leq \left|\left\{(j,k)_{pq}:\left|x_{jk}-x_{p,q}\right|\geq\beta-\alpha\right\}\right|.
\end{equation}
Since $x$ is $I$-statistically pre-Cauchy, so
\begin{center}
$I-\underset{n\rightarrow\infty}{\underset{m\rightarrow\infty}{\lim}}\frac{1}{m^2n^2}\left|\left\{(j,k):\left|x_{jk}-x_{p,q}\right|\geq\beta-\alpha;j,p\leq m;k,q\leq n\right\}\right|=0.$
\end{center}
But from (1) we get,
\begin{center}
$0=L.H.S= I-\underset{n\rightarrow\infty}{\underset{m\rightarrow\infty}{\lim}}D_{mn}(A)D_{mn}(B)=I-\underset{n\rightarrow\infty}{\underset{m\rightarrow\infty}{\lim}}D_{mn}(A)[1-D_{mn}(A)].$
\end{center}
Now from the definition of $I$-convergence it follows that
\begin{center}
$\left\{(m,n)\in\mathbb{N}\times\mathbb{N}:D_{mn}(A)[1-D_{mn}(A)]\geq\frac{1}{25}\right\}\in I.$
\end{center}
Then $\left\{(m,n)\in\mathbb{N}\times\mathbb{N}:D_{mn}(A)[1-D_{mn}(A)]<\frac{1}{25}\right\} = M (say) \in \mathbb{F}(I).$ Clearly for all $(m,n)\in M$ either $D_{mn}(A)<\frac{1}{5}$ or $D_{mn}(A)>\frac{4}{5}$. If $D_{mn}(A)<\frac{1}{5}$ for all $(m,n)\in M_1\subset M$ for some $M_1\in \mathbb{F}(I)$, then we have $I-\underset{n\rightarrow\infty}{\underset{m\rightarrow\infty}{\lim}}D_{mn}(A)=0$. For this, observe that, for given $\epsilon>0$, $0<\epsilon<\frac{1}{5}$, we have from the definition of $I$-convergence $\left\{(m,n)\in\mathbb{N}\times\mathbb{N}:D_{mn}(A)[1-D_{mn}(A)]<\epsilon^2\right\}=M_2(say)\in\mathbb{F}(I).$ Taking $M_0=M_1\cap M_2$, we see that $M_0\in \mathbb{F}(I)$ and $D_{mn}(A)<\epsilon$, for all $(m,n)\in M_0.$ Therefore
\begin{center}
$\left\{(m,n):D_{mn}(A)\geq\epsilon\right\}\subset (\mathbb{N}\times\mathbb{N}\setminus M_0).$
\end{center}
Since $(\mathbb{N}\times\mathbb{N}\setminus M_0)\in I$ so $\left\{(m,n):D_{mn}(A)\geq\epsilon\right\}\in I$ and hence  $I-\underset{n\rightarrow\infty}{\underset{m\rightarrow\infty}{\lim}}D_{mn}(A)=0$.
Similarly if $D_{mn}(A)>\frac{4}{5}$ for all $(m,n)\in M_3\subset M$ for some $M_3\in\mathbb{F}(I)$ then we get $I-\underset{n\rightarrow\infty}{\underset{m\rightarrow\infty}{\lim}}D_{mn}(A)=1.$
\\ If neither of above cases happen then considering dictionary order on $\mathbb{N}\times\mathbb{N}$, we can find an increasing sequence
\begin{center}
$\left\{(m_1,n_1)<(m_2,n_2)<.....\right\}$
\end{center}
from M such that
\begin{eqnarray*}
 D_{m_in_i} &<& \frac{1}{5} ~ when~ i~ is~ an~odd~ integer,\\
        &>& \frac{4}{5} ~ when~ i ~ is~ an ~ even ~integer.
\end{eqnarray*}
Then clearly
\begin{center}
 $P-limsup D_{mn}(A) ~~ - ~~ P-liminf D_{mn}(A)\geq \frac{3}{5}.$
 \end{center}
Again repeating the above process with $\left\{(m,n)\in\mathbb{N}\times\mathbb{N}:D_{mn}(A)[1-D_{mn}(A)]<\frac{1}{36}\right\}=M_4(say)\in \mathbb{F}(I)$ we get, either $I-\underset{n\rightarrow\infty}{\underset{m\rightarrow\infty}{\lim}}D_{mn}(A)=1$ or $I-\underset{n\rightarrow\infty}{\underset{m\rightarrow\infty}{\lim}}D_{mn}(A)=0$ or $P-limsup D_{mn}(A) ~~ - ~~ P-liminf D_{mn}(A)\geq \frac{4}{6}.$

If we continue the repetition of the above procedure, the we see that after a finite number of steps we get,  either  $I-\underset{n\rightarrow\infty}{\underset{m\rightarrow\infty}{\lim}}D_{mn}(A)=0$ or  $I-\underset{n\rightarrow\infty}{\underset{m\rightarrow\infty}{\lim}}D_{mn}(A)=1$. Because if not, then we have
\begin{center}
$P-limsup D_{mn}(A) ~~ - ~~ P-liminf D_{mn}(A)\geq \frac{k-2}{k}$, $k\in\mathbb N$ and $k>4$.
\end{center}
Consequently $P-limsup D_{mn}(A) ~~ - ~~ P-liminf D_{mn}(A)\geq 1 $, which contradicts our hypothesis. This completes the proof of the theorem.
\end{proof}

\begin{rem}
For $A\subset \mathbb{N}\times\mathbb{N}$ if $I-\underset{n\rightarrow\infty}{\underset{m\rightarrow\infty}{\lim}}\frac{1}{mn}\left|\left\{(j,k)\in A:j\leq m, k\leq n\right\}\right|$ exists we say that the double $I$-natural density of $A$ exists and we denote it by $d_I (A)$. Therefore the above result can be re-phrased as:\\
\textit{ Let $x=\left\{x_{jk}\right\}_{j,k\in \mathbb N}$ be a double sequence of real numbers and $(\alpha,\beta)$ is an open interval such that $x_{jk}\notin(\alpha,\beta),$ for all $(j,k)\in \mathbb{N}\times\mathbb{N}$. We write $A=\left\{(j,k):x_{jk}\leq\alpha\right\}$ and further assume that the following property is satisfied
\begin{center}
$P-limsup D_{mn}(A)$-$P-liminf D_{mn}(A)<r.$
\end{center}
for some $0\leq r\leq 1.$ If $x$ is $I$-statistically pre-Cauchy then either $d_I(A)=0$ or $d_I(A)=1.$}
\end{rem}

Before going to our next result, we introduce the following definition.
\begin{defn}
A real number $\xi$ is said to be an $I$-statistical cluster point of a double sequence $x=\left\{x_{jk}\right\}_{j,k\in \mathbb N}$ of real numbers if for any $\epsilon>0$
\begin{center}
$d_I(\left\{(j,k):\left|x_{jk}-\xi\right|<\epsilon\right\})\neq 0.$
\end{center}
\end{defn}
\begin{defn}
A double sequence $ x=\left\{x_{jk}\right\}_{j,k\in\mathbb N}$ is said to be $ I $- statistical bounded if there exists a positive number $ T $ such that for any $ \delta > 0 $ the set $ A = \{(m,n)\in\mathbb{N\times N}: \frac{1}{mn}|\{ (j,k):j\leq m;k \leq n,\|x_{jk}\| \geq T \}|  \geq \delta\}\in I $.
\end{defn}
\begin{lem}
Let $A\subset \mathbb{R}^n$ be a compact set and $A\cap {\Lambda}_x^S(I)=\emptyset$. Then the set $\left\{(j,k)\in\mathbb N\times\mathbb N:x_{jk}\in A\right\}$ has $I$-asymptotic density zero.
\end{lem}
\begin{proof}
Since $A\cap {\Lambda}_x^S(I)=\emptyset$ so for any $\xi\in A$ there is a positive number $\varepsilon=\varepsilon(\xi)$ such that
\begin{center}
 $d_I(\left\{(j,k):\left\|x_{jk}-\xi\right\|<\varepsilon\right\})$.
\end{center}
Let $B_\varepsilon(\xi)=\left\{y\in\mathbb {R}^n:\left\|y-\xi\right\|<\varepsilon\right\}$. Then the set of open sets $\left\{B_\varepsilon(\xi):\xi\in A\right\}$ form an open covers of $A$. Since $A$ is a compact set so there is a finite subcover of $A$, say $\left\{A_i=B_{\varepsilon_i}(\xi_i):i=1,2,..q\right\}$. Then $A\subset\bigcup\limits_{i=1}^{q} A_i$ and 
\begin{center}
$ d_I(\left\{(j,k):\left\|x_{jk}-\xi_i\right\|<\varepsilon_i\right\})=0$ for $i=1,2,...q$.
\end{center}
We can write 
\begin{center}
$\left|\left\{(j,k):j\leq m,~k\leq n; x_{jk}\in A\right\}\right|\leq\sum\limits_{i=1}^{q}\left|\left\{(j,k):j\leq m,~k\leq n; \left\|x_{jk-\xi_i}\right\|<\varepsilon_i\right\}\right|$,
\end{center}
and by the property of $I$-convergence, $
 I \mbox{-}\underset{m,n \rightarrow \infty}{\lim}\left|\left\{(j,k):j\leq m,~k\leq n; x_{jk}\in A\right\}\right|
\leq
\sum\limits_{i=1}^{q}  I \mbox{-}\underset{m,n \rightarrow \infty}{\lim}\left|\left\{(j,k):j\leq m,~k\leq n; \left\|x_{jk-\xi_i}\right\|<\varepsilon_i\right\}\right|=0.$\\
Which gives $d_I(\left\{(j,k):x_{jk}\in A\right\})=0$ and this completes the proof.
\end{proof}
\begin{note}
If the set $A$ is not compact then the above result may not be true. To show this we cite the following example.
\end{note}
\begin{Example}
Let us consider the double sequence $ x = \{ x_{jk}\}_{j,k \in \mathbb{N}}$ in $\mathbb{R}$ defined by
\[ x_{jk} = \left\{
  \begin{array}{l l}
    0, & \quad \text{otherwise }\\
    k, & \quad \text{if k is even}.
  \end{array} \right.\]\\

Then ${\Lambda}_x^S(I)=\left\{0\right\}$. Now if we take $A=[1,\infty)$, then $A\cap{\Lambda}_x^S(I)=\emptyset$, but $d_I(\{(j,k):x_{jk}\in A\})=\frac{1}{2}\neq 0.$
\end{Example}
\begin{thm}
If a double sequence $ x=\left\{x_{jk}\right\}_{j,k\in\mathbb N}\in \mathbb {R}^n$ has a bounded ideal non-thin subsequence, then the set ${\Lambda}_x^S(I)$ is a non-empty closed set. 
\end{thm}

\begin{proof}
Let $ x=\left\{x_{j_pk_q}\right\}_{p,q\in \mathbb N}$ is a bounded ideal non-thin subsequence of $x$, and there is a compact set $A$ such that $x_{jk}\in A$ for each $(j,k)\in P $ where $P=\left\{(j_p,k_q):p,q\in\mathbb N\right\}$. Clearly $d_I(P)=0$. Now if ${\Lambda}_x^S(I)=\emptyset$, then $A\cap{\Lambda}_x^S(I)=\emptyset$ and so by lemma 3.8 we have 
\begin{center}
$d_I(\{(j,k):x_{jk}\in A\})=0$.
\end{center}
But 
\begin{center}
$\left|\left\{(j,k):j\leq m,~ k\leq n,~ (j,k)\in P\right\}\right|\leq \left|\left\{(j,k):j\leq m,~k\leq n,~x_{jk}\in A\right\}\right|$,
\end{center}
which implies that $d_I(P)=0$. This is a contradiction, so ${\Lambda}_x^S(I)\neq \emptyset$.

 Now to show ${\Lambda}_x^S(I)$ is closed, let $\xi$ be a limit point of ${\Lambda}_x^S(I)$. Then for every $\varepsilon>0$ we have $B_\varepsilon(\xi)\cap {\Lambda}_x^S(I)\neq\emptyset$. Let $\beta \in B_\varepsilon(\xi)\cap {\Lambda}_x^S(I)$. Now we can choose $\epsilon'>0$ such that $B_{\epsilon'}(\beta)\subset B_\varepsilon(\xi).$ Since $\beta\in {\Lambda}_x^S(I)$ so 
 \begin{center}
 $d_I(\{(j,k):\left\|x_{jk}-\beta\right\|<\epsilon'\})\neq\emptyset$
 \end{center}

\begin{center}
$\Rightarrow d_I(\{(j,k):\left\|x_{jk}-\xi\right\|<\varepsilon\})\neq\emptyset$.
\end{center}
Hence $\xi\in{\Lambda}_x^S(I)$.
\end{proof}
\begin{defn}
A double sequence $ x=\left\{x_{jk}\right\}_{j,k\in\mathbb N}\in\mathbb {R}^n$ is said to be $ I $- statistically bounded if there exists a compact set $C$ such that for any $ \delta > 0 $ the set $ A = \{(m,n)\in\mathbb{N\times N}: \frac{1}{mn}|\{ (j,k):j\leq m;k \leq n, x_{jk}\notin C \} \geq \delta\}|\in I $.
\end{defn}
\begin{note}
The above Definition 3.4  is compatible with the Definition 3.1 for a double sequence $ x=\left\{x_{jk}\right\}_{j,k\in\mathbb N}\in\mathbb {R}^n$.
\end{note}
\begin{cor}
If $ x=\left\{x_{jk}\right\}_{j,k\in\mathbb N}\in\mathbb {R}^n$ is $I$-statistically bounded. then the set ${\Lambda}_x^S(I)$ is non empty and compact.
\end{cor}

\begin{proof}
Let $C$ be a compact set such that $d_I(\{(j,k):x_{jk}\notin C\})=0.$ Then $d_I(\{(j,k):x_{jk}\in C\})=1$, which implies that $C$ contains a ideal non-thin subsequence of $x$. Hence by Theorem 3.9, ${\Lambda}_x^S(I)$ is nonempty and closed.

 Now to prove that ${\Lambda}_x^S(I)$ is compact it is sufficient to prove that ${\Lambda}_x^S(I)\subset C$. If possible let us assume that $\xi\in {\Lambda}_x^S(I)$ but $\xi\notin C$. Since $C$ is compact, so there exists $\varepsilon>0$ such that $B_\varepsilon(\xi)\cap C=\emptyset$. In this case we have 
 \begin{center}
$\left|\left\{(j,k):\left\|x_{jk}-\xi\right\|\right\}\right|\subset\left|\left\{(j,k):x_{jk}\notin C\right\}\right|$. 
 \end{center}
Therefore $d_I(\left\{(j,k):\left\|x_{jk}-\xi\right\|\right\})=0$, which contradicts the fact that $\xi\in {\Lambda}_x^S(I)$. Therefore ${\Lambda}_x^S(I)\subset C$.
\end{proof}
\begin{thm}
 Let $ x=\left\{x_{jk}\right\}_{j,k\in\mathbb N}\in\mathbb {R}^n$ be $I$-statistically bounded double sequence. Then for every $\varepsilon>0$ the set 
 \begin{center}
$\left\{(j,k):d({\Lambda}_x^S(I),x_{jk})\geq\varepsilon\right\}$ 
 \end{center}
has $I$-asymptotic density zero. Where $d({\Lambda}_x^S(I),x_{jk})=inf_{y\in{\Lambda}_x^S(I)}\left\|y-x_{jk}\right\|$ the distance from $x_{jk}$ to the set ${\Lambda}_x^S(I)$.
\end{thm}
\begin{proof}
Let $C$ be a compact set such that $d_I(\left\{(j,k):x_{jk}\notin C\right\})=0$. Then by Corollary 3.10 we have ${\Lambda}_x^S(I)$ is non-empty and ${\Lambda}_x^S(I)\subset C$.

Now if possible let $d_I(\left\{(j,k):d({\Lambda}_x^S(I),x_{jk})\geq\varepsilon'\right\})\neq 0$ for some $\varepsilon'$.

 Now we define $B_{\varepsilon'}({\Lambda}_x^S(I))=\left\{y\in\mathbb {R}^n:d({\Lambda}_x^S(I),y)<\varepsilon'\right\}$ and let $A=C\setminus B_{\varepsilon'}({\Lambda}_x^S(I))$. Then $A$ is a compact set which contains a ideal non-thin subsequence of $x$. Then by Lemma 3.8 $A\cap {\Lambda}_x^S(I)\neq\emptyset$ this is a contradiction. Hence 
 \begin{center}
\begin{center}
$d_I(\left\{(j,k):d({\Lambda}_x^S(I),x_{jk})\geq\varepsilon\right\})=0.$ 
 \end{center}
 \end{center}

\end{proof}

For the next result we assume that $I$ is such an ideal and $x$ is such that the above result holds without any
additional assumption i.e;\\
\textit{ $(**)$ If $x=\left\{x_{jk}\right\}_{j,k\in\mathbb N}$ is $I$-statistically pre-Cauchy double sequence of real numbers and $x_{jk}\notin(\alpha,\beta)$ for all $(j,k)\in\mathbb{N}\times\mathbb{N}$, where $(\alpha,\beta)$ is an open interval in $\mathbb{R}$, then either $d_I(\left\{(j,k): x_{jk}\leq\alpha\right\})=0$ or $d_I(\left\{(j,k): x_{jk}\geq\beta\right\})=0$}.

\begin{thm}
Let $x=\left\{x_{jk}\right\}$ be an $I$-statistically pre-Cauchy double sequence of real numbers. If the set of limit points of $x$ is no-where dense and $x$ has a $I$-statistical cluster point. Then $x$ is $I$-statistically convergent under the hypothesis $(**)$.
\end{thm}
\begin{proof}
Suppose $x$ has a $I$-statistical cluster point $\xi\in\mathbb R$ . So for any  $\epsilon>0$ we have $d_I(\left\{(j,k):\left|x_{jk}-\xi\right|<\epsilon\right\})\neq 0.$ Suppose that $x$ is $I$-statistically pre-Cauchy satisfying the hypothesis $(**)$ but not $I$-statistically convergent. Then there is an $\epsilon_0>0$ such that  $d_I(\left\{(j,k):\left|x_{jk}-\xi\right|\geq\epsilon_0\right\})\neq 0.$  Without any loss of generality, we assume that $d_I(\left\{(j,k):x_{jk}\leq\xi-\epsilon_0\right\})\neq 0.$ We claim that every point of $(\xi-\epsilon_0,\xi)$ is a limit point of x. If not, then we can
find an interval $(\alpha,\beta)\subset(\xi-\epsilon_0,\xi)$ such that $x_{jk}\notin(\alpha,\beta)$ for all $(j,k)\in\mathbb{N}\times\mathbb{N}$. Thus we have $d_I(\left\{(j,k): x_{jk}\leq\alpha\right\})\neq 0$ and also since $\xi$ is a $I$-statistical cluster point we have $d_I(\left\{(j,k): x_{jk}\geq\beta\right\})\neq 0$. But
this contradicts the hypothesis $(**)$. Hence every point of $(\xi-\epsilon_0,\xi)$ is a limit point of $x$ which contradicts that the set of limit points of $x$ is a nowhere dense set. Hence $x$ is $I$-statistically convergent.
\end{proof}

 \textbf{Acknowledgment}
\\  The second author is thankful to University Grants Commission, New Delhi, India for  his Research fund.


\begin{thebibliography}{99}

\bibitem{1} C. Belen and M. Yildirim, {On generalized statistical convergence of double sequences via
ideals},\textit{ Ann. Univ. Ferrara} 58(1) (2012), 11-20.

\bibitem{2} J. Conor, J. Fridy, and J. Kline, Statistically pre-Cauchy Sequences, \textit{Analysis}, 14(1994) 311-317.

\bibitem{3} P. Das, P. Malik, On Extremal $I$-limit points of double sequences, \textit{Tatra Mt. Math. Publ.} 40 (2008), 91-102.

\bibitem{4} P. Das, , P. Kostyrko, W. Wilczy\'nski and P. Malik, $\mathcal{I}$ and $\mathcal{I}^{*}$
convergence of double sequences, \textit{Math.Slovaca},
58(5)(2008), 605-620.

\bibitem{5} P. Das, E. Savas, S.Kr. Ghosal, On generalizations of certain summability methods using ideals, \textit{Appl. Math. Lett.,}
24(2011) 1509-1514.

\bibitem{6} P. Das, E. Savas, On $I$-statistically pre-Cauchy sequences, \textit{Taiwanese J. Math}.,18(1)(2014) 115-126.
\bibitem{7} 	H. Fast, Sur la convergence statistique. \textit{Colloq. Math} 2(1951) 241-244.
\bibitem{8} P. Malik, A. Ghosh, and M. Maity,  Strong $ I $ and $ I^* $-statistically pre-Cauchy double sequences in Probabilistic Metric Spaces, arXiv:1605.02212 (2016).
\bibitem{Ma} P. Malik, L. K. Dey and P. K. Saha, On statistical cluster point of double sequences.
\bibitem{9} M. Mursaleen, and O.H. Edely, H.H. Osama Statistical convergence of double sequences,\textit{ J. Math. Anal. Appl.},2003, 288, no.1, 223-231.

\bibitem{10}  A. Pringsheim, Zur theorie der Zweifach unendlichen Zahlenfolgen, \textit{Math. Ann., }53, no.(3)(1900) 289-321.
\bibitem{11}  R. F. Patterson Double sequence core theorems, \textit{Internat. J. Math. Math. Sci.} 22(1999), 785–793.
\bibitem{12}  T. S\'alat, P. Kostyrko, W. Wilczy\'nski, $I$-convergence, \textit{Real Anal. Exchange}, 26(2)(2000-2001) 669-685.

\bibitem{13} E. Savas, P. Das, A generalized statistical convergence via ideals, \textit{Appl. Math. Lett.}, 24(2011)
826-830.

\bibitem{14} I. J. Schoenberg, The integrability of certain
functions and related summability methods, \textit{ Amer. Math.
Monthly}, 66(1959), 361-375.


\bibitem{15} H. Steinhus, Sur la convergence ordinatre et la convergence asymptotique, \textit{Colloq. Math}., 2(1951) 73-74.

\bibitem{16}   U. Yamanci, and  M. G$\ddot{u}$rdal, $\mathcal I$-statistically pre-Cauchy double sequences, \textit{Global Journal of Mathematical Analysis}, 2 (4) (2014), 297-303.




\end{thebibliography}
\end{document}